\documentclass{article}

\def\cal{\mathcal}
\usepackage{color,ulem}
\usepackage{amsthm,amssymb,amsfonts,mathrsfs}
\usepackage{amsmath}

\newtheorem{Theorem}{Theorem}[section]
\newtheorem{Proposition}{Proposition}[section]
\newtheorem{Lemma}{Lemma}[section]
\newtheorem{Corollary}{Corollary}[section]
\theoremstyle{definition}

\newtheorem{Definition}{Definition}[section]
\newtheorem{Remark}{Remark}

\newtheorem{Assumptions}{Hypothesis}

\def\supp{\mathop{\rm supp}\nolimits}
\def\R{{\mathbb{R}}}
\def\N{{\mathbb{N}}}

\def\itau12{\int_{\tau_1}^{\tau_2}}

\def\ve{\varepsilon}

\def\ds{\displaystyle}

\newcommand{\cY}{{\mathcal Y}}
\newcommand{\cA}{{\mathcal A}}

\usepackage[colorlinks=true,urlcolor=red,linkcolor=blue,citecolor=red,bookmarks=red]{hyperref}

\title{Linear stabilization for a degenerate wave equation in non divergence form with drift}

\author{
{\sc Genni Fragnelli}\thanks{The author is a member of the  {\it Gruppo Nazionale per l'Analisi Ma\-te\-matica, la Probabilit\`a e le loro Applicazioni (GNAMPA)} of the Istituto Nazionale di Alta Matematica (INdAM) and a member of {\it UMI ``Modellistica Socio-Epidemiologica (MSE)''}. She is supported by the FFABR {\it Fondo per il finanziamento delle attivit\`a base di ricerca} 2017 and  by the DEB.HORIZON$_{-}$EU$_{-}$DM737 project 2022 {\it COntrollability of PDEs in the Applied Sciences (COPS)}. She also thanks the Project Horizon Europe Seeds 2021 {\it STEPS: STEerability and controllability of PDES in Agricultural and Physical models}, CUP: H91I21001640006.}\\
Department  of Ecological and Biological Sciences\\ Tuscia University\\ Largo dell'Universit\`a, 01100 Viterbo - Italy\\ email: genni.fragnelli@unitus.it\\
{\sc Dimitri Mugnai}\thanks{The author is a member of the Gruppo Nazionale per l'Analisi Ma\-te\-matica, la Probabilit\`a e le loro Applicazioni (GNAMPA) of the Istituto Nazionale di Alta Matematica (INdAM). Supported by the INdAM-GNAMPA Project 2022 "PDE ellittiche a diffusione mista", by the FFABR {\it Fondo per il finanziamento delle attivit\`a base di ricerca} 2017 and by the DEB.HORIZON$_{-}$EU$_{-}$DM737 project 2022 {\it COntrollability of PDEs in the Applied Sciences (COPS)}.}\\
Department  of Ecological and Biological Sciences\\ Tuscia University\\ Largo dell'Universit\`a, 01100 Viterbo - Italy\\ email:
dimitri.mugnai@unitus.it}

\date{}
\begin{document}

\maketitle

\vspace{0.3cm}

\begin{abstract}
We consider a degenerate wave equation in one dimension, with drift and in presence of a leading operator which is not in divergence form. We impose a homogeneous Dirichlet boundary condition where the degeneracy occurs and a boundary damping at the other endpoint. We provide some conditions for the uniform exponential decay of solutions for the associated Cauchy problem.
\end{abstract}

Keywords: degenerate wave equation, drift,  stabilization, exponential decay

MSC: 35L80, 93D23, 93D15

\section{Introduction}

It is well known that the displacement of a mass subjected to the action of spring is modelled by a nonlinearly damped oscillator and the displacement $u$ of
the mass is described by the scalar equation
\[
u''+h(u')+ ku+f(u)=0.
\]
Here $k>0$ is a physical parameter, $f$ is continuous and describes some nonlinear phenomenon, while $h(u')$ is the nonlinear damping. 
Set
\[
F(u)=\int_0^uf(t)dt
\]
and define the energy associated to a solution $u$ as
\[
E_u(t)=\frac{1}{2}(u'(t))^2+\frac{1}{2}ku^2(t)+F(u(t)).
\]
 Assuming that
\begin{equation}\label{condh}
sh(s)\geq 0 \quad \forall\, s\in \R,
\end{equation}
then
\[
E_u'(t)=-u'(t)h(u'(t))\leq0,
\]
which is a dissipation relation.

Similarly, consider a vibrating membrane fixed on the
boundary. Then the evolution of the displacement of a point $x$ of the membrane at
time $t$ is described by the wave equation
\[
u_{tt}-\Delta u+h(u_t)+f(u)=0 \qquad \mbox{in }(0,\infty)\times \Omega,
\]
where $\Omega$ is a (possibly bounded) domain of $\R^N$. The energy of a
solution $u$ is 
\[
E_u(t)=\int_\Omega \left[\frac{1}{2}\left(u_t^2+|\nabla u|^2\right)+F(u)\right]dx
\]
and if everything is smooth enough and $h$ is as before, we get
\[
E_u'(t)=-\int_\Omega u_th(u_t)dx\leq0.
\]
An analogous result can be proved if the wave equation is coupled with a damping in a portion $\Gamma$ of $\partial \Omega$ as
\[
\frac{\partial u}{\partial \nu}+h(u_t)=0 \qquad \mbox{ in }[0,\infty)\times \Gamma.
\]

>From a monotonicity property of the energy, it is natural to look for conditions which guarantee the {\em stability} of even more general problems, namely
\[
\|\nabla u\|_{L^2} +\|u_t\|_{L^2}\to 0 \mbox{ as }t\to\infty.
\]
In this respect, several conditions have been found, even when \eqref{condh} is not satisfied, also in presence of nonlinear sources (see, for instance, \cite{fmdcdss}, \cite{sicon}, \cite{HMV}, \cite{MV2002}, \cite{PS} and the two monographs \cite{alabaumon} and \cite{bc}).

However, all the previous results concern {\em nondegenerate} problems. On the other hand, the standard linear theory for transverse waves in a string leads to the classical wave equation
\[
\rho(x) u_{t t}(t,x)=\frac{\partial \mathcal T}{\partial x}(t,x) u_{x}(t,x)+\mathcal  T(t,x)u_{xx}(t,x),
\]
where $u(t,x)$ is the vertical displacement of the string from the $x$ axis at position $x$ and time $t$, $\rho(x)$ is the mass density of the string at position $x$, while $\mathcal T(t,x)$ denotes the tension in the string at position $x$ and time $t$. Dividing by $\rho(x)$, assuming $\mathcal T$ is independent of $t$, and setting  $a(x)=\mathcal T(x)\rho^{-1}(x)$, $b(x)=\mathcal T'(x)\rho^{-1}(x)$, we obtain
\[
u_{t t}(t,x)=a(x) u_{xx}(t,x)+b(x)u_{x}(t,x).
\]
If density is extremely large at some point, for instance $x=0$,  we can assume $a(0)=0$. The drift term $b$ may degenerate at $x=0$, as well.

A related equation in divergence form and without drift, namely
\[
u_{t t}(t,x)=(a(x) u_{x})_x(t,x),
\]
has been studied in \cite{alabau} (see also the arxiv version 2015) and \cite{ZG} for a general $a$ and for $a(x)=x^K$, $K\in(0,2)$, respectively. In both cases, boundary controllability was pursued via multiplier methods (\cite{alabau}) or spectral methods (\cite{ZG}). Moreover, in \cite{alabau}, stability results were proved, as well. Inspired by such a case, here we consider the  problem
\begin{equation}\label{wave}
\begin{cases}
y_{tt} - a(x)y_{xx} -b(x)y_x=0, & (t,x) \in Q_T,
\\
y_t(t,1)+\eta y_x(t,1)+ \beta y(t,1)=0,& t\in (0,T),
\\
y(t, 0)= 0,  & t >0,
\\
y(0,x)= y_0(x), \quad y_t(0,x)=y_1(x),& x \in (0, 1),
\end{cases}
\end{equation}
where $Q_T=(0,T) \times(0,1)$, $T>0$,  $a, b\in C^0[0,1]$, with $a>0$ on
$(0,1]$, $a(0)=0$ and  $\ds \frac{b}{a} \in L^1(0,1)$: hence, if $a(x)=x^K$, $K>0$, we can consider $b(x)=x^h$ for any $h>K-1$.  In the boundary term we take $\beta \ge0$ and  $\eta$ is the well-known absolutely continuous weight function
\[
\eta(x):=\exp\left \{\int_{\frac{1}{2}}^x\frac{b(s)}{a(s)}ds\right \}
, \quad x\in [0,1],
\]
introduced by Feller in a related contex \cite{F1} and used by
several authors, see, e.g., \cite{cfr}, \cite{favya} and \cite{mp}.  
Finally, the initial data $u_0$ and $u_1$ belong to suitable weighted spaces.

The main feature in this problem is that $a$ degenerates at $x=0$ (with $b$ possibly degenerate, as well) and that the leading operator is not in the usual divergence form. As a consequence, classical methods cannot be used directly to study such a problem and different approach is needed, see the following section.

As for the function $a$, we consider two cases: $a$ can be weakly degenerate or strongly degenerate. More precisely, we have the following standard definition.

\begin{Definition}
A function $a$ is {\it weakly degenerate at $0$}, (WD) for short, if $a\in C^0[0,1]\cap C^1(0,1]$ is such that
$a(0)=0$, $a>0$ on $(0,1]$ and, if 
\begin{equation}\label{stima_a}
K:=\sup_{x \in (0,1]}\frac{x|a'(x)|}{a(x)},
\end{equation}
then $K\in (0,1)$.
\end{Definition}

\begin{Definition}
A function $a$ is {\it strongly degenerate at $0$}, (SD) for short, if $a\in C^1[0,1]$ is such that
$a(0)=0$, $a>0$ on $(0,1]$ and  in
\eqref{stima_a} we have $K\in [1,2)$.
\end{Definition}

In the previous definition we always assume that $K<2$, since it is essential in Hypothesis \ref{Ass2} below.

\begin{Remark}
Observe that \eqref{stima_a} implies that the function
\begin{equation}\label{crescente}
x \mapsto \frac{x^\gamma}{a(x)}
\end{equation}
is nondecreasing in $(0,1]$ for all $\gamma \ge K$. In particular, Hypothesis \ref{Ass0} below is satisfied. Moreover,
\begin{equation}\label{limite}
\lim_{x\rightarrow 0}  \frac{x^\gamma}{a(x)} =0
\end{equation}
for all $\gamma >K$ and
\begin{equation}\label{defM}
\left|\frac{x^\gamma b(x)}{a(x)}\right| \le \frac{1}{a(1)} \|b\|_{L^\infty(0,1)}
\end{equation}
for all $\gamma \ge K$, assuming $b \in L^\infty(0,1)$.
\end{Remark}

\section{Preliminary results and well posedness}\label{section2}

In this section we introduce the functional setting needed to treat our problem. However, our assumptions here are more general than those required to get the desired stability and have an independent interest.

We start assuming a very modest requirement, which will be assumed throughout the paper.
\begin{Assumptions}\label{basic}
Functions $a$ and $b$ are continuous in $[0,1]$ and such that
$\dfrac{b}{a}\in L^1(0,1)$.
\end{Assumptions}

\begin{Remark}\label{Rem2}$\ $\\
1. We notice that, at this stage, $a$ may not degenerate at $x=0$. However, if it is (WD) 
then $\ds\frac{1}{a}\in L^1(0,1)$ and $b$ can not degenerate. If $a$ is (SD) then $\ds\frac{1}{a} \notin L^1(0,1)$, hence the assumption $\ds\dfrac{b}{a}\in L^1(0,1)$ implies $b(0)=0$.  In this case $b$ can be (WD) or (SD).\\
2. If  $a$ is (WD) or (SD) with $K=1$ then \eqref{defM} immediately implies that $\ds\frac{xb}{a}$ is bounded.
\end{Remark}

If Hypothesis \ref{basic} holds, it is clear that the function $\eta:[0,1]\to \R$ introduced before is well defined and we immediately find that $\eta\in C^0[0,1]\cap
C^1(0,1]$ is a strictly positive function, which is {\sl bounded above and below by a positive constant}. Notice also that $\eta$ can be extended to a function of class $C^1[0,1]$ when $b$ degenerates at 0 not slower than $a$, for instance if $a(x)=x^K$ and $b(x)=x^h$ with $K\leq h$.

Now, we are ready to go back to problem \eqref{wave} and study its well-posedness. To do that, let us
define
\[
\sigma(x):=\frac{a(x)}{\eta(x)},
\]
which is a continuous function in $[0,1]$, independently of the possible degeneracy of $a$. Moreover, observe that if $y$ is a sufficiently smooth function, e.g. $y \in
W^{2,1}_{\text{loc}}(0,1)$, then we can write
\begin{equation}\label{defsigma}
Ay:=ay''+by'
\end{equation}
as
\[
 Ay= \sigma(\eta y')'.
\]

Following \cite{cfr}, let us consider
the following Hilbert spaces with the related inner products:
\[
 L^2_{\frac{1}{\sigma}}(0,1) :=\left\{ u \in L^2(0,1)\; \big|\; \|u\|_{ \frac{1}{\sigma}}<\infty \right\},
 \;  \langle u,v\rangle_{\frac{1}{\sigma}}:= \int_0^1u v\frac{1}{\sigma}dx,
\]
for every $u,v \in L^2_{\frac{1}{\sigma}}(0,1)$;
\[
H^1_{\frac{1}{\sigma}}(0,1) :=L^2_{\frac{1}{\sigma}}(0,1)\cap
H^1(0,1),
\;  \langle u,v\rangle_1 :=   \langle u,v\rangle_{\frac{1}{\sigma}} + \int_0^1\eta  u'v'dx,\]  
for every $u,v \in H^1_{\frac{1}{\sigma}}(0,1)$ and
\[
H^2_{\frac{1}{\sigma}}(0,1) := \Big\{ u \in
H^1_{\frac{1}{\sigma}}(0,1)\; \big|\;Au \in
L^2_{\frac{1}{\sigma}}(0,1)\Big\},
\;   \langle u,v\rangle_2 := \langle u,v\rangle_1+  \langle Au,Av\rangle_{\frac{1}{\sigma}},
\]
for every $u,v \in H^2_{\frac{1}{\sigma}}(0,1)$. The previous inner products obviously induce the related respective norms
\[
\|u\|^2_{\frac{1}{\sigma}} = \int_0^1 \frac{u^2}{\sigma}dx, \quad
\|u\|^2_{1,\frac{1}{\sigma}} = \|u\|^2_{\frac{1}{\sigma}} + \int_0^1\eta (u')^2 dx
\]
and
\[
\|u\|^2_{2, \frac{1}{\sigma}} = \|u\|^2_{1,\frac{1}{\sigma}} + \int_0^1 \sigma [(\eta u')']^2dx.
\]
 Moreover, consider the spaces
\[
H^1_{\frac{1}{\sigma},0}(0,1) := \{ u \in H^1_{\frac{1}{\sigma}}(0,1) : u(0)=0\}
\]
and
\[
H^2_{\frac{1}{\sigma},0}(0,1):= \{ u \in H^2_{\frac{1}{\sigma}}(0,1) : u(0)=0\},
\]
endowed with the previous inner products and related norms.

\begin{Remark}
Being $\eta$ bounded {\it above} and {\it below} by positive constants, it is clear that the inner product
with weight $\eta$ is equivalent to the standard one without weight. However, it will be clear soon that in this way we have a quite better functional setting, as Corollary \ref{equivalenze} below will show.
\end{Remark}

\begin{Assumptions}\label{Ass0}
Hypothesis \ref{basic} holds. In addition, $a$ is such that $a(0)=0$, $a>0$ on
$(0,1]$ and there exists $K>0$  such that  the function
\[
 x \longmapsto\dfrac{x^K}{a(x)}
\]
is nondecreasing in a right neighborhood of $x=0$. 
\end{Assumptions}

Notice that here we require only continuity on $a$ (and no differentiability); moreover, the monotonicity property required only near $0$ holds globally in $(0,1]$ if $a$ is (WD) or (SD).

Proceeding as in \cite{bfmu} and using the fact that $v(0)=0$ for all $v\in H^1_{\frac{1}{\sigma},0}(0,1)$, one has
 \begin{Proposition}[Hardy-Poincar\'e Inequality]\label{H1a} $\,$
 Assume Hypothesis $\ref{Ass0}$. Then, there exists
 $C_{HP}>0$ such that
 \begin{equation}\int_0^1 v^2 \frac{1}{\sigma} dx
 \le C_{HP} \int_0^1(v')^2dx\quad \forall \; v\in H^1_{\frac{1}{\sigma},0}(0,1).
 \end{equation}
 \end{Proposition}

In particular, we have  the equivalence below.
\begin{Corollary}
\label{equivalenze}Assume Hypothesis $\ref{Ass0}$. Then the two norms
$
\|u\|_{1,\frac{1}{\sigma}}^2
$
and
\[
\|u\|_1^2:= \int_0^1 (u')^2dx, 
\]
are equivalent for all  $u \in H^1_{\frac{1}{\sigma},0}(0,1)$. In particular,
\[
\|u\|_1^2 \le \frac{1}{\min_{[0,1]} \eta} \|u\|_{1,\frac{1}{\sigma}}^2 \quad \text{and} \quad  \|u\|_{1,\frac{1}{\sigma}}^2 \le (C_{HP}+\max_{[0,1]} \eta) \|u\|_1^2,
\]
where $C_{HP}$ is the Hardy-Poincar\'e constant introduced in Proposition $\ref{H1a}$.
\end{Corollary}

Now, define the domain $D(A)$ of
the operator $A$ given in \eqref{defsigma} as
\[
D(A) = H^2_{{\frac{1}{\sigma},0}}(0,1).
\]

We start with a set of results, actually of independent interest, which describe the functional setting we shall use.

\begin{Lemma}\label{intparti} Assume Hypothesis $\ref{basic}$.
For all  \:$(u,v)\in D(A)\times
H^1_{{\frac{1}{\sigma},0}}(0,1)$ one has
\begin{equation} \label{Green}
<Au,v>_{\frac{1}{\sigma}}= -\int_0^1\eta u'v'dx+ (\eta u' v)(1).
\end{equation}
\end{Lemma}
\begin{proof}
As a first step, we consider the space $H^1_c(0,1):=\big\{v\in
H^1(0,1)\:|\:\supp\{v\}\subset (0,1]\big\}$.

As in the proof of \cite[Lemma 2.1]{cfr1}, we can see that  $H^1_c(0,1)$ is dense in
$H^1_{{\frac{1}{\sigma},0}}(0,1)$. Indeed, fix $v\in
H^1_{{\frac{1}{\sigma},0}}(0,1)$ and consider the sequence
$(v_n)_{n \ge 3}$, where $v_n: =\xi_nv\in H^1_c(0,1)$ and
\[
\xi_n(x):=\left\{
\begin{array}{ll}
0,&x\in\; \left[0,1/n\right],\\
1,&x\in \;  \left[2/n, 1\right],
\\
nx-1, &x\in\;  \left( 1/n ,2/n \right);
\end{array}
\right.
\]
then  $v_n\rightarrow v$ in
$H^1_{{\frac{1}{\sigma},0}}(0,1)$.

Now, as in \cite{cfr}, consider
$$\varPhi(v):=\int_0^1\big((au'' +bu')v\frac{1}{\sigma} + \eta
u'v'\big)dx-(\eta u' v)(1),$$ with $u\in H^2_{{\frac{1}{\sigma},0}}(0,1)$. Then,
$\varPhi$ is a bounded linear functional on
$H^1_{{\frac{1}{\sigma},0}}(0,1)$. Moreover, $\varPhi=0$ on
$H^1_c(0,1)$. Indeed, taking $v\in H^1_c(0,1)$,  one has  that
\[
\int_0^1(au'' + bu')v\frac{1}{\sigma}dx
 =
 \int_0^1 \sigma(\eta u')'v\frac{1}{\sigma}dx
 =
-\int_0^1\eta u'v'dx + (\eta u' v)(1).
\]
Thus, $\varPhi=0$ on $H^1_{{\frac{1}{\sigma},0}}(0,1)$, that is
\eqref{Green} holds.
\end{proof}


To evaluate boundary terms the following results are important. Since the proofs are similar to those of \cite[Lemma 3.2]{bfmu}, we postpone them to the Appendix.
\begin{Lemma}\label{lemmalimite} 
\begin{enumerate}
\item Assume Hypothesis $\ref{basic}$. If $y \in H^2_{\frac{1}{\sigma}}(0,1)$ and if $u \in H^1_{\frac{1}{\sigma},0}(0,1)$, then $\lim_{x \rightarrow 0} u(x) y'(x)=0$.
\item Assume Hypothesis $\ref{Ass0}$. If $u\in D(A)$, then $xu'(\eta u')'\in L^1(0,1)$.
\item Assume Hypothesis $\ref{Ass0}$. If $u\in D(A)$ and $K\leq 1$, then $\lim_{x \rightarrow 0} x (u'(x))^2=0$.
\item Assume Hypothesis $\ref{Ass0}$. If  $u \in D(A)$,  $K>1$ and $\ds\frac{xb}{a} \in L^\infty(0,1), $ then $\ds\lim_{x\rightarrow 0} x  (u'(x))^2=0$.
\item Assume  Hypothesis $\ref{Ass0}$. If  $u \in H^1_{\frac{1}{\sigma}}(0,1)$, then   $\ds
\lim_{x\rightarrow 0} \frac{x}{a}u^2(x)=0.$
\end{enumerate}
\end{Lemma}

The last result, which will be crucial to obtain the stabilization of problem \eqref{wave}, is given by the following proposition.
\begin{Proposition}\label{Prop2.2} Assume Hypothesis $\ref{Ass0}$  and for $\beta \ge 0$ define
\[|||z|||_1^2:= \int_0^1 \eta (z')^2dx + \beta z^2(1)
\] 
for all  $z \in H^1_{\frac{1}{\sigma},0}(0,1)$. Then the two norms $|||\cdot|||_1$ and $\|\cdot\|_1$ are equivalent. Moreover, for every $\lambda \in \R$, the variational problem
\begin{equation}\label{varionalproblem}
\int_0^1 \eta z'\phi' dx+ \beta z(1)\phi(1) = \lambda \phi(1) \quad \forall \; \phi \in H^1_{\frac{1}{\sigma},0}(0,1)
\end{equation}
admits a  unique solution $z \in H^1_{\frac{1}{\sigma},0}(0,1)$ which satisfies the estimates
\begin{equation}\label{02}
|||z|||_1^2\le \frac{\lambda^2}{\min_{[0,1]}\eta} \quad \text{and} \quad \|z\|^2_{{\frac{1}{\sigma}}}\le \frac{ \max_{[0,1]} \eta+ C_{HP}}{\min_{[0,1]}^2 \eta}\lambda^2,
\end{equation}
where $C_{HP}$ is the Hardy-Poincaré constant in Proposition $\ref{H1a}$.
In addition, $z \in H^2_{\frac{1}{\sigma},0}(0,1)$ and solves
\begin{equation}\label{VP}
\begin{cases}
&- \sigma (\eta z_x)_x=0,\\
&\eta z_x(t,1) +\beta z(t,1)=\lambda.
\end{cases}
\end{equation}
\end{Proposition}
\begin{proof}
Observe that
\begin{equation}\label{*z}
|z(1)|= \left| \int_0^1 z'(t)dt \right| \le \|z\|_1,
\end{equation} 
for all  $z\in H^1_{\frac{1}{\sigma},0}(0,1)$. Thus, $|||\cdot|||_1$ and $\|\cdot\|_1$ are equivalent. Indeed,
for all $z \in H^1_{\frac{1}{\sigma},0}(0,1)$
\begin{equation}\label{04*}
\|z\|_1^2 \le \frac{1}{\min_{[0,1]} \eta} |||z|||_1^2;
\end{equation}
moreover, since $\beta z^2(1)\le \beta \|z\|_1^2$ by \eqref{*z}, one has
\[
|||z|||_1^2\le (\max_{[0,1]}\eta + \beta) \|z\|_1^2,
\]
and the claim holds.

Now, consider the bilinear and symmetric form  $\Lambda: H^1_{\frac{1}{\sigma},0}(0,1) \times H^1_{\frac{1}{\sigma},0}(0,1) \rightarrow \R$, given by
\[
\Lambda (z, \phi) := \int_0^1 \eta z' \phi' dx + \beta z(1)\phi(1).
\]
for all  $z,\phi \in H^1_{\frac{1}{\sigma},0}(0,1)$.
Clearly $\Lambda$ is also  coercive and continuous. Indeed, by Corollary \ref{equivalenze}
\[
\Lambda(z,z) = \int_0^1 \eta (z')^2dx + \beta z^2(1)\ge  \int_0^1 \eta (z')^2dx \ge \frac{\min_{[0,1]}\eta}{C_{HP}+\max_{[0,1]} \eta} \|z\|_{1,\frac{1}{\sigma}}^2.
\]
Moreover,
\[
|\Lambda (z, \phi)| \le\max_{[0,1]}\eta \|z'\|_{L^2(0,1)}\|\phi'\|_{L^2(0,1)}+ \beta |z(1)||\phi(1)|.
\]
By \eqref{*z} applied to $z$ and $\phi$, one has
\[
|\Lambda (z, \phi)| \le (\max_{[0,1]}\eta +\beta) \|z\|_1 \|\phi\|_1.
\]
Now, consider the linear functional
\[
\mathcal L( \phi):= \lambda \phi(1),
\]
with $\phi \in H^1_{\frac{1}{\sigma},0}(0,1)$. Clearly, $\mathcal L$ is continuous and linear. Thus, by the Lax-Milgram Theorem, there exists a unique solution $z \in H^1_{\frac{1}{\sigma},0}(0,1)$ of
\begin{equation}\label{05}
\Lambda (z, \phi)= \mathcal L (\phi)
\end{equation}
for all $\phi \in H^1_{\frac{1}{\sigma},0}(0,1)$.
In particular,
\begin{equation}\label{04}
\Lambda (z,z) =\int_0^1 \eta (z')^2dx+ \beta z^2(1) =  \mathcal L(z)= \lambda z(1).
\end{equation}
By \eqref{04} and \eqref{04*}, we have
\[
|||z|||_1^2 = \lambda z(1) \le \frac{|\lambda|}{\sqrt{\min_{[0,1]}\eta}} |||z|||_1;
\]
thus
\[
|||z|||_1\le \frac{|\lambda|}{\sqrt{\min_{[0,1]}\eta}} \quad \text{and} \quad |||z|||_1^2\le \frac{\lambda^2}{\min_{[0,1]}\eta}.
\]
Moreover, by Corollary \ref{equivalenze}, we know that in $ H^1_{\frac{1}{\sigma},0}(0,1)$ the two norms 
$
\|\cdot\|_1$ and $\|\cdot\|_{1, \frac{1}{\sigma}}$
are equivalent. Thus
\[
\begin{aligned}
|||z|||^2_1 &\ge \min_{[0,1]}\eta \|z\|_1^2 + \beta z^2(1) \ge  \min_{[0,1]} \eta\|z\|_1^2\\ & \ge \frac{ \min_{[0,1]} \eta}{ \max_{[0,1]} \eta+ C_{HP}}\|z\|_{1, \frac{1}{\sigma}}^2 \ge \frac{ \min_{[0,1]} \eta}{ \max_{[0,1]} \eta+ C_{HP}}\|z\|^2_{L^2_{\frac{1}{\sigma}}(0,1)}.
\end{aligned}
\]
Thus, by \eqref{02},
\[
\|z\|^2_{{\frac{1}{\sigma}}} \le \frac{ \max_{[0,1]} \eta+ C_{HP}}{ \min_{[0,1]} \eta}|||z|||^2_1 \le \frac{\max_{[0,1]} \eta+ C_{HP} }{\min_{[0,1]}^2 \eta}\lambda^2.
\]

Now, we will prove that $z \in H^2_{\frac{1}{\sigma},0}(0,1)$ solves \eqref{VP}. To this aim, we consider again \eqref{05}. Since it holds for every $\phi \in H^1_{\frac{1}{\sigma},0}(0,1)$, it holds in particular for every $\phi \in C_c^\infty(0,1)$, so that 
\[
\int_0^1 \eta z'\phi'=0 \mbox{ for all }\phi \in C_c^\infty(0,1).
\]
By the fundamental lemma of the calculus of variations (for instance, see \cite[Lemma 1.2.1]{JLJ}), we get that
$\eta z'$ is constant a.e. in $(0,1)$ and so $(\eta z')' =0$ a.e. in $(0,1)$; in particular 
\[
\sigma (\eta z')' =0 \quad \text{ a.e.  in} \; (0,1)
\]
and so
$Az= \sigma (\eta z')' \in L^2_{\frac{1}{\sigma}}(0,1)$.

Now, coming back to \eqref{05}, we have
\[
\int_0^1 \eta z'\phi' dx+ \beta z(1)\phi(1)=\lambda \phi(1) \Longleftrightarrow [\eta z'\phi]_{x=0}^{x=1} + \beta z(1)\phi(1) = \lambda \phi(1)
\]
for all $\phi \in H^1_{\frac{1}{\sigma},0}(0,1)$. Thus, since $\phi(0)=0$, we obtain
\[
(\eta z')(1) + \beta z(1)=\lambda,
\]
that is $z$ solves \eqref{VP}.
\end{proof}


We are now ready to study the well posedness of problem \eqref{wave}. For this, we introduce the Hilbert space
\[
\mathcal H_0 := H^1_{{\frac{1}{\sigma},0}}(0,1)\times L^2_{{\frac{1}{\sigma}}}(0,1),  
\]
with the inner product
\[
\langle (u, v), (\tilde u, \tilde v) \rangle_{\mathcal H_0}:=  \int_0^1 u'\tilde u'dx + \int_0^1 v\tilde v\frac{1}{\sigma}dx+ \beta u(1) \tilde u(1) 
\]
for every $(u, v), (\tilde u, \tilde v)  \in \mathcal H_0$, and the induced norm
\[
\|(u,v)\|_{\mathcal H_0}^2:=  \int_0^1 (u')^2dx + \int_0^1 v^2\frac{1}{\sigma}dx + \beta u^2(1).
\]
 Observe that if $u \in H^1_{{\frac{1}{\sigma},0}}(0,1)$, then $u$ is continuous, so that $u(1)$ is well defined. Moreover, being $\eta \in C^0[0,1] \cap C^1(0,1]$ far away from 0,  for every $(u, v), (\tilde u, \tilde v)  \in \mathcal H_0$, the norm $\|(u,v)\|_{\mathcal H_0}^2$ is equivalent to 
\[
\|(u,v)\|_1^2:=  \int_0^1\eta (u')^2dx + \int_0^1 v^2\frac{1}{\sigma}dx + \beta u^2(1).
\]
Obviously, to such a norm we associate the inner product
\[
\langle (u, v), (\tilde u, \tilde v) \rangle_1:=  \int_0^1 \eta u'\tilde u'dx + \int_0^1 v\tilde v\frac{1}{\sigma}dx+ \beta u(1) \tilde u(1),
\]
which we will use from now on, being more convenient for our treatment.

Now, consider the
matrix operator $\cA : D(\cA) \subset \mathcal H_0 \rightarrow \mathcal H_0$, given by
\[
\cA:= \begin{pmatrix} 0 & Id\\
A&0 \end{pmatrix},\]
and
\[D(\cA):= \{(u,v) \in H^2_{{\frac{1}{\sigma},0}}(0,1) \times H^1_{{\frac{1}{\sigma},0}}(0,1): (\eta u')(1) + v(1)+ \beta u(1)=0\}. \]
Thus, by using the operator $(\mathcal A, D(\mathcal A))$, we rewrite \eqref{wave} as a Cauchy problem. Indeed, setting, as usual,
\[
\cY(t):= \begin{pmatrix} y\\ y_t \end{pmatrix} \; \text{ and }\; \cY_0:= \begin{pmatrix}y_0\\y_1 \end{pmatrix},
\]
one has that  \eqref{wave} can be rewritten  
as 
\begin{equation}\label{CP}
\begin{cases}
\dot \cY (t)= \cA \cY (t), & t \ge 0,\\
\cY(0) = \cY_0.
\end{cases}
\end{equation}
If we prove that $(\cA, D(\cA))$  generates a contraction semigroup $(S(t))_{t \ge 0}$ and
 $\cY_0  \in \mathcal H_0$, then $\cY(t)= S(t)\cY_0$ gives the mild solution of \eqref{CP}. 
The next theorem holds.
\begin{Theorem}\label{generator}
Assume Hypothesis $\ref{Ass0}$. Then the operator $(\cA, D(\cA))$ is non positive with dense domain and generates a contraction semigroup  $(S(t))_{t \ge 0}$. 
\end{Theorem}
For the proof of this theorem we use the next result
\begin{Theorem}[\cite{nagel}, Corollary 3.20]\label{densità}
Let $(\mathcal A,D(\mathcal A))$ be a dissipative operator on a reflexive Banach space such that $\lambda I-\mathcal A$ is surjective for some $\lambda >0$. Then $\mathcal A$ is densely defined and generates a contraction semigroup.
\end{Theorem}
\begin{proof}[Proof of Theorem $\ref{generator}$]
According to the previous theorem, it is sufficient to prove that $ \mathcal A:D(\mathcal A)\to \mathcal H_0$ is dissipative and that $I-\mathcal A$ is surjective.

\underline{$ \mathcal A$ is dissipative:} take $(u,v) \in D(\mathcal A)$. Then $(u,v) \in H^2_{{\frac{1}{\sigma},0}}(0,1) \times H^1_{{\frac{1}{\sigma},0}}(0,1)$ and so \eqref{Green} holds. Hence, by Lemma \ref{intparti},
\[
\begin{aligned}
\langle \mathcal A (u,v), (u,v) \rangle_1 &=\langle (v, Au), (u,v) \rangle _1\\
&=\int_0^1 \eta u'v'dx+ \int_0^1 vAu\frac{1}{\sigma}dx + \beta v(1)u(1)
\\&
=\int_0^1\eta u'v'dx -\int_0^1\eta u'v'dx+ (\eta u' v)(1) + \beta v(1)u(1)\\
&= v(1)((\eta u')(1)+ \beta u(1)) \\
&= - v^2(1) \le 0.
\end{aligned}
\]
\underline{$I - \mathcal A$ is surjective:} 
take  $(f,g) \in \mathcal H_0=H^1_{{\frac{1}{\sigma},0}}(0,1)\times L^2_{{\frac{1}{\sigma}}}(0,1)$. We have to prove that there exists $(u,v) \in D(\mathcal A)$ such that
\begin{equation}\label{4.3'}
 ( I-\mathcal A)\begin{pmatrix} u\\
v\end{pmatrix} = \begin{pmatrix}f\\
g \end{pmatrix} \Longleftrightarrow  \begin{cases} v= u -f,\\
-Au + u= f+ g.\end{cases}
\end{equation}
Thus, define $F: H^1_{{\frac{1}{\sigma},0}}(0,1) \rightarrow \R$ as
\[
F(z)=\int_0^1(f+g) z\frac{1}{\sigma}  dx +  z(1)f(1).
\]
Obviously, $F\in H^{-1}_{{\frac{1}{\sigma},0}}(0,1)$, the dual space of $H^1_{{\frac{1}{\sigma},0}}(0,1)$ with respect to the pivot space $L^2_{{\frac{1}{\sigma}}}(0,1)$: indeed,  $f\in H^1_{{\frac{1}{\sigma},0}}(0,1)$, and so also $f(1)$ is well defined, and $g\in L^2_{{\frac{1}{\sigma}}}(0,1)$. Now, introduce the bilinear form $L:H^1_{{\frac{1}{\sigma},0}}(0,1)\times H^1_{{\frac{1}{\sigma},0}}(0,1)\to \R$ given by
\[
L(u,z):=  \int_0^1 u z \frac{1}{\sigma} dx + \int_0^1\eta u'z'dx +(\beta+1) u(1)z(1)
\]
for all $u, z \in H^1_{{\frac{1}{\sigma},0}}(0,1)$. Clearly, since $\beta \ge 0$, $L(u,z)$ is coercive. Moreover $L(u,z)$ is 
 continuous: indeed, for all $u \in H^1_{{\frac{1}{\sigma},0}}(0,1)$, as in \eqref{*z},
\[
|u(1)|\le \int_0^1| u'(t)|dt  = \|u'\|_{L^1(0,1)} \le \|u'\|_{L^2(0,1)};
\]
thus, for all $u, z \in H^1_{{\frac{1}{\sigma},0}}(0,1)$,
\[
|L(u,z)| \le  \|u\|_{ L^2_{\frac{1}{\sigma}} (0,1) }\|z\|_ {L^2_{\frac{1}{\sigma}} (0,1) } +(\|\eta\|_{L^\infty(0,1)} + \beta +1)\|u'\|_{L^2(0,1)}\|z'\|_ {L^2(0,1)},
\]
and the conclusion follows from Corollary \ref{equivalenze}.

As a consequence, by the Lax-Milgram Theorem, there exists a unique solution $u \in H^1_{{\frac{1}{\sigma},0}}(0,1)$ of
\[
L(u,z)= F(z)  \mbox{ for all }z\in H^1_{{\frac{1}{\sigma},0}}(0,1),\]
namely
\begin{equation}\label{4.4}
\int_0^1 u z \frac{1}{\sigma} dx + \int_0^1 \eta u'z'dx +(\beta +1)u(1)z(1)= \int_0^1(f+g) z \frac{1}{\sigma} dx +  z(1)f(1)
\end{equation}
for all $z \in H^1_{{\frac{1}{\sigma},0}}(0,1)$.

Now, take $v:= u-f$; then $v \in H^1_{{\frac{1}{\sigma},0}}(0,1)$.
We will prove that $(u,v) \in D(\mathcal A)$ and solves \eqref{4.3'}. To begin with, \eqref{4.4} holds for every $z \in C_c^\infty(0,1).$ Thus we have
\[
\int_0^1 \eta u'z'dx = \int_0^1(f+g-u) z \frac{1}{\sigma} dx 
\]
 for every $z \in C_c^\infty(0,1).$ Hence $\ds-(\eta u')'= (f+g- u)  \frac{1}{\sigma}$ a.e. in $(0,1)$. This implies  that $-\sigma(\eta u')'=( f+g- u) \in L^2_{ \frac{1}{\sigma}}(0,1)$, i.e. $Au \in L^2_{ \frac{1}{\sigma}}(0,1)$; thus $u \in D(A)$. 
 Moreover, coming back to \eqref{4.4} and thanks to \eqref{Green},
 \[
 - \int_0^1 \sigma(\eta u')'z \frac{1}{\sigma} dx + (\eta u' z)(1) +(\beta+1) u(1)z(1)= \int_0^1(f+g- u) z \frac{1}{\sigma} dx + z(1)f(1).
 \]
Using the fact that $-\sigma(\eta u')'= (f+g- u) $ a.e. in $(0,1)$, we obtain
\[
(\eta u' z)(1) +(\beta +1)u(1)z(1)=  z(1)f(1) 
\]
 for all $z \in H^1_{{\frac{1}{\sigma},0}}(0,1)$. Hence 
$
\eta (1) u' (1) +(\beta+1) u(1) - f(1)=0 .
$
Recalling that $v = u - f$, one has
\[
\eta (1) u' (1) +\beta u(1) + v(1)=0.
\]
In conclusion,  $(u,v) \in D(\mathcal A)$, $ u -A u =f+g$ and $v=u-f$, i.e. $(u,v)$ solves \eqref{4.3'}.
\end{proof}

As usual in semigroup theory, the mild solution of \eqref{CP} obtained above can be more regular: if $\cY_0 \in D(\mathcal A)$, then the solution is classical, in the sense that $\cY \in  C^1([0, +\infty); \mathcal H_0) \cap C([0, +\infty);D(\mathcal A))
$ and the equation in \eqref{wave} holds for all $t \ge0$. Hence, as in \cite[Corollary 4.2]{alabau} or in \cite[Proposition 3.15]{daprato}, one has the following theorem.
\begin{Theorem}\label{esistenza}
Assume Hypothesis $\ref{Ass0}$.

If $(y_0,y_1) \in \mathcal H_0$, then there exists a unique mild solution
\[
y \in C^1([0, +\infty); L^2_{\frac{1}{\sigma}}(0,1)) \cap C([0, +\infty); H^1_{\frac{1}{\sigma},0} (0,1))
\]
of \eqref{wave} which depends continuously on the initial data $(y_0,y_1) \in  \mathcal H_0.$
Moreover, if  $(y_0,y_1) \in D(\mathcal A)$, then the solution $y$ is classical, in the sense that
\[
y \in C^2([0, +\infty); L^2_{\frac{1}{\sigma}}(0,1)) \cap C^1([0, +\infty); H^1_{\frac{1}{\sigma},0} (0,1)) \cap C([0, +\infty); H^2_{\frac{1}{\sigma},0} (0,1))
\]
and the equation of \eqref{wave} holds for all $t \ge 0$. 
\end{Theorem}

\section{The stability result}\label{section3}
In this section we prove the main result of the paper when $a$ is (WD) or (SD). Actually, we will prove that the energy associated to the initial problem is nonincreasing and, in particular, it decreases exponentially under suitable assumptions.

To this aim, let $y$ be a mild solution of \eqref{wave} and consider its energy, given by 
\begin{equation}\label{def_energy}
E_y(t)=\frac{1}{2}\left[\int_0^1 \left(\frac{1}{\sigma}y_t^2(t,x)  +\eta y_x^2(t,x)\right)dx + \beta y^2(t,1)\right], \quad \; t \ge 0.
\end{equation}
With this definition in hand,  one can prove that the energy  is nonincreasing.
\begin{Theorem}\label{Energiadecrescente}
Assume Hypothesis $\ref{Ass0}$ and let $y$ be a classical solution of \eqref{wave} (for instance, if $(y_0,y_1) \in D(\mathcal A)$). Then 
the energy is nonincreasing and
\[
\frac{dE_y(t)}{dt}=-y_t(t,1)^2, \quad t\geq0.
\]
\end{Theorem}
\begin{proof}
By multiplying the equation by $\ds\frac{y_t}{\sigma}$, integrating over $(0,1)$ and using the boundary conditions one has
\[
\begin{aligned}
0&=\frac{1}{2} \int_0^1\frac{d}{dt} \left(\frac{y^2_t}{\sigma}\right)dx -[\eta y_xy_t]_{x=0}^{x=1} +\int_0^1 \eta y_xy_{tx}dx\\
&=\frac{1}{2}\frac{d}{dt}\left[\int_0^1	\left( \frac{y^2_t}{\sigma}+\eta y^2_x\right)dx + \beta y^2(t,1)  \right] + y^2_t(t,1)\\
&=\frac{1}{2}\frac{d}{dt} E_y(t)+ y^2_t(t,1).
\end{aligned}
\]
Indeed, taking $u=y_t$ in Lemma \ref{lemmalimite}.1, one has $\lim_{x\rightarrow 0} (\eta y_xy_t)(t,x)=0$.
Hence,
\[
\frac{1}{2}\frac{d}{dt} E_y(t) = -  y^2_t(t,1)\le 0
\]
for all $t \ge 0$.
\end{proof}

Now,  our aim is to estimate the energy $E_y(t)$ with the value of the energy $E_y(0)$ at $t=0$. To do that, we need to restrict Hypothesis \ref{Ass0}, requiring the crucial assumption below.

\begin{Assumptions}\label{Ass1}
Hypothesis \ref{basic} holds and $a$ is (WD) or (SD); if $K>1$, then also assume that $\ds\frac{xb}{a} \in L^\infty(0,1).$
\end{Assumptions}

\begin{Remark}
It is clear that, by definition of (WD) or (SD), if Hypothesis \ref{Ass1} holds, then Hypothesis \ref{Ass0} holds, as well.
\end{Remark}

The next preliminary result holds.
\begin{Proposition}\label{Prop1}
 Assume Hypothesis $\ref{Ass1}$ and let $y$ be a classical solution of \eqref{wave}. Then 
\begin{equation}\label{1new}
\begin{aligned}
0&= 2\int_0^1\left[\frac{xy_xy_t}{\sigma}\right]_{t=s}^{t=T} dx  -\frac{1}{\sigma(1)}\int_s^T y_t^2(t,1)dt - \eta(1) \int_s^T  y_x^2(t,1)dt\\&- \int_{Q_s} x \eta \frac{b}{a}  y_x^2 dxdt + \int_{Q_s}\left(1-\frac{x(a'-b)}{a}\right) \frac{1}{\sigma}y_t^2 dxdt+ \int_{Q_s} \eta y_x^2 dxdt,
\end{aligned}
\end{equation}
for every $T>s>0$.
\end{Proposition}
\begin{proof}
Take $s \in (0,T)$; then, multiplying the equation of \eqref{wave} by $\ds\frac{xy_x}{\sigma}$, integrating over $Q_s:= (s,T) \times (0,1)$ and recalling \eqref{defsigma}, we have
\begin{equation}\label{*}
\begin{aligned}
0&= \int_{Q_s} \frac{y_{tt}xy_x}{\sigma} dxdt- \int_{Q_s}x(\eta y_x)_xy_xdxdt\\
&= \int_0^1 \left[\frac{xy_xy_t}{\sigma}\right]_{t=s}^{t=T} dx - \int_{Q_s} \frac{xy_{xt}y_t}{\sigma}dxdt-\int_{Q_s}x(\eta' y_x +\eta y_{xx})y_x dxdt\\
&=\int_0^1\left[\frac{xy_xy_t}{\sigma}\right]_{t=s}^{t=T} dx -\frac{1}{2} \int_{Q_s} \frac{x}{\sigma}(y_t^2)_xdxdt-\int_{Q_s}x\eta' y_x^2 dxdt -\frac{1}{2} \int_{Q_s} x\eta (y_{x}^2)_x dxdt\\
&= \int_0^1\left[\frac{xy_xy_t}{\sigma}\right]_{t=s}^{t=T} dx-\frac{1}{2}\int_s^T \left[\frac{x}{\sigma}y_t^2\right]_{x=0}^{x=1} dt +\frac{1}{2}  \int_{Q_s}\left(\frac{x}{\sigma}\right)'y_t^2dxdt-\int_{Q_s}x\eta \frac{b}{a} y_x^2 dxdt  \\
&  -\frac{1}{2}\int_s^T \left[x\eta y_x^2\right]_{x=0}^{x=1}dt+\frac{1}{2} \int_{Q_s}(x\eta)' y_x^2dxdt.
\end{aligned}
\end{equation}
Recalling the definition of $\eta$ and $\sigma$, we immediately find
\begin{equation}\label{1}
\begin{aligned}
0&= \int_0^1\left[\frac{xy_xy_t}{\sigma}\right]_{t=s}^{t=T} dx  -\frac{1}{2}\int_s^T \left[\frac{x}{\sigma}y_t^2\right]_{x=0}^{x=1} dt - \frac{1}{2}\int_s^T \left[x\eta y_x^2\right]_{x=0}^{x=1}dt\\
&- \frac{1}{2} \int_{Q_s} x \eta \frac{b}{a}  y_x^2 dxdt +\frac{1}{2} \int_{Q_s}\left(1-\frac{x(a'-b)}{a}\right) \frac{1}{\sigma}y_t^2 dxdt+ \frac{1}{2} \int_{Q_s} \eta y_x^2 dxdt
\end{aligned}
\end{equation}
for all $s \in (0,T)$. 

Now, we consider the boundary terms.
Thanks to the boundary conditions of $y$, \eqref{limite} and Lemma \ref{lemmalimite}, we immediately have that
\[
\lim_{x \rightarrow 0}\frac{x}{\sigma}y_t^2(t,x)=\lim_{x \rightarrow 0} \frac{x}{a}\eta y_t^2(t,x)= 0
\]
and
\[
\lim_{x \rightarrow 0}x\eta y_x^2(t,x)=0.
\]
Hence, \eqref{1} multiplied by $2$ gives \eqref{1new}.
\end{proof}

\begin{Proposition}\label{Prop2}
 Assume Hypothesis $\ref{Ass1}$ and let $y$ be a classical solution of \eqref{wave}. Then, for all $T>s>0$ we have
\begin{equation}\label{sopraBT}
\begin{aligned}
&\int_{Q_s}\left(1-\frac{x(a'-b)}{a}+\frac{K}{2}\right) \frac{1}{\sigma}y_t^2 dxdt +
\int_{Q_s} \left(1-x\frac{b}{a}-\frac{K}{2}\right)\eta y_x^2 dxdt= (B.T.)
\end{aligned}
\end{equation}
where 
\begin{equation}\label{BT}
(B.T.)=\int_0^1\left[-2x\frac{y_xy_t}{\sigma}   + \frac{K}{2} \frac{yy_t}{\sigma}\right]_{t=s}^{t=T} dx +\int_s^T \left[\frac{1}{\sigma}y_t^2+\eta y_x^2 -\frac{K}{2} \eta yy_x\right](t,1) dt
\end{equation}
and $Q_s:= (s,T) \times (0,1)$.
\end{Proposition}
\begin{proof}
By multiplying the equation in \eqref{wave} by $\ds\frac{y}{\sigma}$ and integrating over $Q_s$, we have
\begin{equation}\label{2}
\begin{aligned}
\int_{Q_s}\left(-\frac{y_t^2}{\sigma} + \eta y_x^2\right) dxdt +  \int_0^1\left[\frac{yy_t}{\sigma}\right]_{t=s}^{t=T} dx -  \int_s^T\left[\eta y_xy\right]_{x=0}^{x=1} dt=0.
\end{aligned}
\end{equation}
Using the fact that $y$ is a classical solution of \eqref{wave}, that $y(t,0)=0$ and Lemma \ref{lemmalimite}, one has 
\[
  \int_s^T\left[\eta y_xy\right]_{x=0}^{x=1} dt =   \int_s^T [\eta y_xy](t,1) dt;
\]
thus, multiplying \eqref{2} by  $\ds\frac{K}{2}$, one has
 \begin{equation}\label{2new}
\begin{aligned}
\frac{K}{2}\int_{Q_s}\left(-\frac{y_t^2}{\sigma} + \eta y_x^2\right) dxdt + \frac{K}{2} \int_0^1\left[\frac{yy_t}{\sigma}\right]_{t=s}^{t=T} dx -\frac{K}{2} \int_s^T\left[\eta y_xy\right](t,1) dt=0.
\end{aligned}
\end{equation}

By summing \eqref{2new} and \eqref{1new}, we get the claim.
\end{proof}

\begin{Proposition}\label{Prop3}
 Assume Hypothesis $\ref{Ass1}$, $\beta\ge0$ and let $y$ be a classical solution of \eqref{wave}. Then, for any $T>s>0$ and for every $\delta >0$ we have
\begin{equation}\label{Stimabo}
\begin{aligned}
\int_s^Ty^2(t,1)dt &\le \left(2+  \frac{2C_{HP}}{\min_{[0,1]}^3\eta}+\frac{1}{\delta}+\frac{1}{\delta}\frac{\max_{[0,1]}\eta+ C_{HP}}{\min_{[0,1]}^2 \eta} \right)E_y(s) \\
&+  2\delta \left(\frac{1}{\min_{[0,1]}^3\eta} + 1 \right)\int_s^T E_y(t) dt.
\end{aligned}
\end{equation}
\end{Proposition}
\begin{proof}
To prove the statement, fix $t \in [s,T]$ and set $\lambda =y(t,1)$ and let $z=z(t,\cdot)$ be the unique solution of 
\[
\int_0^1 \eta z'\phi' dx+ \beta z(1)\phi(1) = \lambda \phi(1) \quad \forall \; \phi \in H^1_{\frac{1}{\sigma},0}(0,1).
\]
By Proposition \ref{Prop2.2}, $z(t,\cdot) \in H^2_{\frac{1}{\sigma},0}(0,1)$ for all $t$ and solves
\begin{equation}\label{problem1}
\begin{cases}
&- \sigma (\eta z_x)_x=0,\\
&\eta z_x(t,1) +\beta z(t,1)=\lambda.
\end{cases}
\end{equation}
Now, 
multiply the equation in \eqref{wave}  by $\ds \frac{z}{\sigma}$ and integrate over $Q_s$. Then, we have
\[
\begin{aligned}
0&= \int_{Q_s} \left(y_{tt} \frac{z}{\sigma} -(\eta y_x)_x z\right)dxdt = \int_0^1 \left[ y_t\frac{z}{\sigma}\right]_{t=s}^{t=T} dx\\
&- \int_{Q_s} y_t \frac{z_t}{\sigma}dxdt - \int_s^T\left[ \eta y_xz\right]_{x=0}^{x=1}dt + \int_{Q_s} \eta y_xz_xdxdt.
\end{aligned}
\]
By Lemma \ref{lemmalimite}
\[
\lim_{\epsilon \rightarrow 0} (\eta y_xz)(t, \epsilon)=0,
\]
thus we have
\begin{equation}\label{star1}
\int_0^1 \left[ y_t\frac{z}{\sigma}\right]_{t=s}^{t=T} dx- \int_{Q_s} y_t \frac{z_t}{\sigma}dxdt = \int_s^T \eta y_x(t,1) z(1)dt -\int_{Q_s} \eta y_xz_xdxdt.
\end{equation}
By multiplying the equation in \eqref{problem1} by $\ds \frac{y}{\sigma}$ and integrating on $Q_s$, one has
\[
\int_{Q_s}  (\eta z_x)_x y dxdt=0.
\]
Using the fact that $(\eta z_x y)(t,0)=0$ and $\ds (\eta z_x)(t,1)= \lambda -\beta z(t,1)$, we get
\[
\begin{aligned}
\int_s^T\left[ \eta z_xy\right]_{x=0}^{x=1}dt & - \int_{Q_s} \eta z_x y_x dxdt=0 \Longleftrightarrow  \int_s^T( \eta z_xy)(t,1)dt= \int_{Q_s} \eta z_x y_x dxdt \\
&\Longleftrightarrow   \int_s^T( \lambda - \beta z(t,1))y(t,1) dt = \int_{Q_s} \eta z_x y_x dxdt.
\end{aligned}
\]
Substituting in \eqref{star1} and recalling that $y$ solves \eqref{wave} and $\lambda = y(t,1)$, we have
\[
\begin{aligned}
\int_0^1\left[\frac{y_tz}{\sigma}\right]_{t=s}^{t=T}dx &- \int_{Q_s} \frac{y_t z_t}{\sigma} dxdt =\int_s^T(\eta y_xz)(t,1) dt - \int_s^T (\lambda -\beta z (t,1))y(t,1) dt\\
&=\int_s^T (\eta y_x z)(t,1) dt -\int_s^T y^2(t,1) dt + \beta \int_s^T (zy)(t,1) dt\\
&
= - \int_s^Ty_t (t,1) z(t,1)dt - \int_s^Ty^2(t,1) dt.
\end{aligned}\]
Then
\begin{equation}\label{3pezzetti}
\int_s^Ty^2(t,1)dt = \int_{Q_s}\frac{y_tz_t}{\sigma}dxdt - \int_s^T(y_t z)(t,1)dt - \int_0^1 \left[\frac{y_t z}{\sigma}\right]_{t=s}^{t=T}dx.
\end{equation}
Hence, to bound $\int_s^Ty^2(t,1)dt$, we estimate the last three terms in the previous equality. By Proposition \ref{H1a}, \eqref{02} and recalling that $\lambda= y(t,1)$, we have
\[
\begin{aligned}
\int_0^1\left| \frac{y_t z}{\sigma}(\tau, x)\right|dx&\le \frac{1}{2} \int_0^1 \frac{y_t^2(\tau, x)}{\sigma} dx + \frac{1}{2} \int_0^1 \frac{z^2(\tau, x)}{\sigma}dx \\
&\le  \frac{1}{2}\int_0^1 \frac{y_t^2(\tau, x)}{\sigma} dx +\frac{1}{2}\frac{C_{HP}}{\min_{[0,1]}\eta} \int_0^1 (z_x^2\eta)(\tau, x) dx \\
& \le \frac{1}{2} \int_0^1 \frac{y_t^2(\tau, x)}{\sigma} dx +\frac{1}{2}\frac{C_{HP}}{\min_{[0,1]}^2\eta}y^2(t,1)\\
& \le E_y(\tau) + \frac{1}{2}\frac{C_{HP}}{\min_{[0,1]}^2\eta}y^2(t,1).
\end{aligned}
\]
for all $\tau \in [s,T]$. 
By \eqref{*z}, one has
\begin{equation}\label{y(1)}
\frac{1}{2}y^2(t,1) \le \frac{1}{2}\frac{1}{\min_{[0,1]}\eta} \int_0^1 (y_x^2 \eta)(t,x) dx \le \frac{1}{\min_{[0,1]}\eta} E_y(t),
\end{equation}
for all $ t \in [s,T]$. Thus, by Theorem \ref{Energiadecrescente},
\[
\int_0^1\left| \frac{y_t z}{\sigma}(\tau, x)\right|dx\le E_y(\tau) + \frac{C_{HP}}{\min_{[0,1]}^3\eta}E_y(t)\le \left(1+ \frac{C_{HP}}{\min_{[0,1]}^3\eta}\right) E_y(s).
\]
Thus, again by Theorem \ref{Energiadecrescente},
\begin{equation}\label{step1}
\left|  \int_0^1 \left[\frac{y_t z}{\sigma}\right]_{t=s}^{t=T}dx \right| \le 2\left(1+ \frac{C_{HP}}{\min_{[0,1]}^3\eta}\right) E_y(s).
\end{equation}

Moreover, for any $\delta >0$ we have
\begin{equation}\label{7.1}
\int_s^T | (y_tz)(t,1) |dt \le \frac{1}{\delta} \int_s^T  y_t^2 (t,1)dt + \delta \int_s^T z^2(t,1)dt.
\end{equation}
By  \eqref{*z},  \eqref{02} and \eqref{y(1)} one has
\[
\begin{aligned}
z^2(t,1) &\le \frac{1}{\min_{[0,1]}\eta}\int_0^1 (\eta z_x^2)(t,x) dx \le \frac{1}{\min_{[0,1]}\eta}|||z|||_1^2 \\
&\le \frac{y^2(t,1)}{\min_{[0,1]}^2\eta} \le \frac{2}{\min_{[0,1]}^3\eta} E_y(t).
\end{aligned}
\]
Thus, by \eqref{7.1} and Theorem \ref{Energiadecrescente}, we have
\begin{equation}\label{step2}
\begin{aligned}
\int_s^T |(y_tz)(t,1)| dt & \le \frac{1}{\delta} \int_s^T  y_t^2 (t,1)dt + \delta \frac{2}{\min_{[0,1]}^3\eta} \int_s^T E_y(t) dt  \\
& \le - \frac{1}{\delta} \int_s^T \frac{dE_y(t)}{dt} dt + \delta \frac{2}{\min_{[0,1]}^3\eta} \int_s^T E_y(t) dt  \\
& \le \frac{E_y(s)}{\delta} +  \frac{2 \delta }{ \min_{[0,1]}^3\eta} \int_s^T E_y(t) dt.
\end{aligned}
\end{equation}

Finally, we estimate the first integral in \eqref{3pezzetti}. To this aim, consider again problem \eqref{VP} and differentiate with respect to $t$. Then
\[
\begin{cases}
&- \sigma (\eta z_{tx})_x=0,\\
&\eta z_{tx}(t,1) +\beta z_t(t,1)=\lambda_t =y_t(t,1).
\end{cases}
\]
Clearly, $z_t$ satisfies the estimates in \eqref{02}; in particular
\[
|||z_t|||^2_1 \le \frac{y_t^2(t,1)}{\min_{[0,1]} \eta} \quad \text{and} \quad \|z_t\|^2_{{\frac{1}{\sigma}}} \le \frac{\max_{[0,1]} \eta+ C_{HP}}{\min_{[0,1]}^2 \eta} y_t^2(t,1).
\]
Thus, for $\delta >0$, we find
\begin{equation}\label{step3}
\begin{aligned}
\int_{Q_s}\left|\frac{y_t z_t}{\sigma}\right|dxdt &\le \delta \int_{Q_s} \frac{y_t^2}{\sigma}dxdt + \frac{1}{\delta}\int_{Q_s} \frac{z_t^2}{\sigma}dxdt\\
&\le 2\delta \int_s^TE_y(t)dt + \frac{1}{\delta}\frac{\max_{[0,1]} \eta+ C_{HP}}{\min_{[0,1]}^2 \eta}\int_s^T y_t^2(t,1)dt\\
&= 2\delta \int_s^TE_y(t)dt - \frac{1}{\delta}\frac{\max_{[0,1]} \eta+ C_{HP}}{\min_{[0,1]}^2 \eta}\int_s^T\frac{dE_y(t)}{dt}dt\\
& \le  2\delta \int_s^TE_y(t)dt+ \frac{1}{\delta}\frac{\max_{[0,1]} \eta+ C_{HP}}{\min_{[0,1]}^2 \eta}E_y(s).
\end{aligned}
\end{equation}
Hence, going back to \eqref{3pezzetti}, by \eqref{step1}, \eqref{step2} and \eqref{step3}, we get
\[
\begin{aligned}
\int_s^Ty^2(t,1)dt &\le 2\left(1+ \frac{C_{HP}}{\min_{[0,1]}^3\eta}\right)E_y(s)  +\frac{E_y(s)}{\delta} +  \frac{2 \delta }{ \min_{[0,1]}^3\eta} \int_s^T E_y(t) dt\\
&+2\delta \int_s^TE_y(t)dt+ \frac{1}{\delta}\frac{\max_{[0,1]} \eta+ C_{HP}}{\min_{[0,1]}^2 \eta}E_y(s)\\
&= \left(2+  \frac{2C_{HP}}{\min_{[0,1]}^3\eta}+\frac{1}{\delta}+\frac{1}{\delta}\frac{\max_{[0,1]}\eta+ C_{HP}}{\min_{[0,1]}^2 \eta} \right)E_y(s) \\
&+  2\delta \left(\frac{1}{\min_{[0,1]}^3\eta} + 1 \right)\int_s^T E_y(t) dt,
\end{aligned}
\]
as claimed.
\end{proof}

Now, we assume an additional hypothesis on the functions $a$ and $b$.
\begin{Assumptions}\label{Ass2}
Hypothesis \ref{Ass1} holds and there exists $\ve_0>0$ such that 
\[
(2-K)a-2x|b|\geq \ve_0a \mbox{ for every }x\in [0,1].
\]
\end{Assumptions}

Notice that  the required inequality implies that $|b(x)|\leq \frac{(2-K -\ve_0)a}{2x}$. Thus, if $a(x)=x^K$, the condition reads $|b(x)|\leq \frac{(2-K -\ve_0)x^{K-1}}{2}$. Of course,  the previous condition is automatically satisfied in absence of the drift.

\begin{Proposition}\label{Prop4}
 Assume Hypothesis $\ref{Ass2}$, $\beta \ge 0$ and let $y$ be a classical solution of \eqref{wave}. Then, for any $T>s>0$
\begin{equation}\label{tondonew}
\begin{aligned}
\frac{\ve_0}{2} \int_{Q_s} \left(\frac{y_t^2}{\sigma} + \eta y_x^2 \right)dxdt&
\le  4 \Theta E_y(s)  + \left( \frac{1}{\sigma(1)} + \frac{1}{\eta(1)}+\frac{\beta}{\eta(1)}+ \frac{K}{4} \right)(E_y(s)-E_y(T)) \\
&+  \left(\frac{\beta^2}{\eta(1)}+ \frac{K\beta}{2}+\frac{\beta}{\eta(1)}+ \frac{K}{4}\right) \int_s^Ty^2(t,1)dt,
\end{aligned}
\end{equation}
where 
\begin{equation} \label{Theta}\ds \Theta:= \max\left\{\frac{1}{a(1)}+K\frac{C_{HP}}{\min_{[0,1]}\eta }; 1+ \frac{K}{4} \right\}.
\end{equation}
\end{Proposition}
\begin{proof}
By all the assumptions in force,
\[
1-\frac{x(a'-b)}{a}+\frac{K}{2}=\frac{(2-K)a+2(Ka-xa')+2xb}{2a}\geq \frac{\ve_0}{2},
\]
and
\[
\left(1-x\frac{b}{a}-\frac{K}{2}\right)\geq \frac{\ve_0}{2}
\]
as well. Thus the boundary terms given in \eqref{BT}, from \eqref{sopraBT} can be estimated by below in the following way:
\begin{equation}\label{star}
(B.T.) \ge  \frac{\ve_0}{2} \int_{Q_s} \left(\frac{y_t^2}{\sigma} + \eta y_x^2 \right)dxdt.
\end{equation}

Now, we estimate the boundary terms from above.  First of all, consider the term 
\[
\int_0^1\left( -2x\frac{y_xy_t}{\sigma} + \frac{K}{2} \frac{yy_t}{\sigma}\right) (\tau, x) dx
\]
for all $\tau \in [s,T]$. Using the fact that $\ds \frac{x^2}{a(x)} \le \frac{1}{a(1)}$ by \eqref{crescente}, together with Proposition \ref{H1a}, one has
\[
\begin{aligned}
&\int_0^1\left( -2x\frac{y_xy_t}{\sigma} + \frac{K}{2} \frac{yy_t}{\sigma}\right)(\tau, x) dx \le
\\
&\le \int_0^1 \frac{x^2y_x^2\eta}{a}(\tau, x) dx+ \int_0^1\frac{y_t^2}{\sigma} (\tau, x)dx+ \frac{K}{4}\int_0^1 \frac{y_t^2}{\sigma}(\tau, x)dx + \frac{K}{4}\int_0^1 \frac{y^2}{\sigma}(\tau, x)dx\\
&\le \frac{1}{a(1)}\int_0^1 \eta y_x^2(\tau, x) dx + \left(1 +\frac{K}{4}\right)\int_0^1 \frac{y_t^2}{\sigma}(\tau, x)dx +  K\frac{C_{HP}}{\min_{[0,1]}\eta }\int_0^1y_x^2 \eta (\tau, x)dx\\
&\le \max\left\{\frac{1}{a(1)}+ K\frac{C_{HP}}{\min_{[0,1]}\eta }; 1+ \frac{K}{4} \right\} \left(2E_y(\tau) -\beta y^2(\tau,1)\right).
\end{aligned}
\]
By \eqref{star} and \eqref{BT}, thanks to Theorem \ref{Energiadecrescente}, we get
\begin{equation}\label{tondo}
\begin{aligned}
\frac{\ve_0}{2} \int_{Q_s} \left(\frac{y_t^2}{\sigma} + \eta y_x^2 \right)dxdt&\le 2\Theta E_y(T) + 2\Theta E_y(s) + \int_s^T\left(\frac{y_t^2}{\sigma} + \eta y_x^2 -\frac{K}{2}\eta y y_x\right)(t,1)dt\\
& \le 4 \Theta E_y(s) + \int_s^T\left(\frac{y_t^2}{\sigma} + \eta y_x^2 -\frac{K}{2}\eta y y_x\right)(t,1)dt.
\end{aligned}
\end{equation}
where $\ds \Theta$ is defined in \eqref{Theta}.

Now, we need to estimate
$
\int_s^Th(t)dt,
$
where
\[
h(t) := \left(\frac{y_t^2}{\sigma} + \eta y_x^2 -\frac{K}{2}\eta y y_x\right)(t,1).
\]
To this aim, recall that $\eta y_x(t,1)= -\beta y(t,1) -y_t(t,1)$. Thus
\[
\begin{aligned}
h(t)&= \frac{y_t^2}{\sigma}(t,1) + \frac{1}{\eta(1)}\left(  -\beta y(t,1) -y_t(t,1)\right)^2 - \frac{K}{2} \big( -\beta y(t,1) -y_t(t,1)\big) y(t,1)\\
& = \frac{y_t^2}{\sigma}(t,1) +  \frac{1}{\eta(1)}\big( \beta^2 y^2(t,1)+ y_t^2(t,1)+ 2\beta y(t,1)y_t(t,1) \big)\\
&+ \frac{K}{2} \beta y^2(t,1) + \frac{K}{2}y_t(t,1)y(t,1)\\
&=\left( \frac{1}{\sigma(1)} + \frac{1}{\eta(1)}\right) y_t^2(t,1) + \beta \left( \frac{\beta}{\eta(1)}+ \frac{K}{2}\right) y^2(t,1) + \left(\frac{2\beta}{\eta(1)}+ \frac{K}{2} \right)yy_t(t,1). 
\end{aligned}
\]

Thus, by Theorem \ref{Energiadecrescente},
\begin{equation}\label{tondo1}
\begin{aligned}
\int_s^T h(t)dt &=  \left( \frac{1}{\sigma(1)} + \frac{1}{\eta(1)}\right) \int_s^Ty_t^2(t,1)dt + \beta\left( \frac{\beta}{\eta(1)}+ \frac{K}{2}\right) \int_s^Ty^2(t,1)dt \\
& +\left(\frac{2\beta}{\eta(1)}+ \frac{K}{2} \right)\int_s^Tyy_t(t,1) dt\\
&\le \left( \frac{1}{\sigma(1)} + \frac{1}{\eta(1)}\right) \int_s^T-\frac{dE_y}{dt}dt + 
\beta\left( \frac{\beta}{\eta(1)}+ \frac{K}{2}\right) \int_s^Ty^2(t,1)dt\\
& + \frac{1}{2}\left(\frac{2\beta}{\eta(1)}+ \frac{K}{2} \right)\int_s^T y^2(t,1)dt +\frac{1}{2}\left(\frac{2\beta}{\eta(1)}+ \frac{K}{2} \right)\int_s^T-\frac{dE_y}{dt}dt\\
&=   \left( \frac{1}{\sigma(1)} + \frac{1}{\eta(1)}+\frac{\beta}{\eta(1)}+ \frac{K}{4} \right)(E_y(s)-E_y(T)) \\
&+  \left(\frac{\beta^2}{\eta(1)}+ \frac{K\beta}{2}+\frac{\beta}{\eta(1)}+ \frac{K}{4} \right) \int_s^Ty^2(t,1)dt.
\end{aligned}
\end{equation}
By \eqref{tondo} and \eqref{tondo1}, we have
\[
\begin{aligned}
\frac{\ve_0}{2} \int_{Q_s} \left(\frac{y_t^2}{\sigma} + \eta y_x^2 \right)dxdt&
\le  4 \Theta E_y(s)  + \left( \frac{1}{\sigma(1)} + \frac{1}{\eta(1)}+\frac{\beta}{\eta(1)}+ \frac{K}{4} \right)(E_y(s)-E_y(T)) \\
&+  \left(\frac{\beta^2}{\eta(1)}+ \frac{K\beta}{2}+\frac{\beta}{\eta(1)}+ \frac{K}{4} \right) \int_s^Ty^2(t,1)dt
\end{aligned}
\]
and \eqref{tondonew} holds.
\end{proof}

As a consequence of Propositions \ref{Prop3} and \ref{Prop4}, we can state the main result of the paper. As a first step, define
\[C_1:=\left(\frac{\beta^2}{\eta(1)}+ \frac{K\beta}{2}+\frac{\beta}{\eta(1)}+ \frac{K}{4} + \frac{\beta\ve_0}{2}\right) \left(\frac{1}{\min_{[0,1]}^3\eta} + 1 \right).\]
\begin{Theorem}\label{Energiastima}
 Assume Hypothesis $\ref{Ass2}$, $\beta \ge 0$ and let $y$ be a mild solution of \eqref{wave}. Then, for all $t >0$ and for all $\ds\delta \in \left(0, \frac{\ve_0}{2C_1}\right)$, we have
\begin{equation}\label{estiener}
E_y(t) \le E_y(0) e^{1- \frac{t}{M}},
\end{equation}
where 
\begin{equation}\label{M}
\begin{aligned}M&:=\frac{1}{C_{\ve_0, \beta, K, \eta}}\left(4\Theta + \frac{1}{\sigma(1)} + \frac{1}{\eta(1)}+\frac{\beta}{\eta(1)}+ \frac{K}{4}\right)\\
&+ \frac{1}{C_{\ve_0, \beta, K, \eta}}\left(\frac{\beta^2}{\eta(1)}+ \frac{K\beta}{2}+\frac{\beta}{\eta(1)}+ \frac{K}{4} + \frac{\beta\ve_0}{2}\right) \cdot\\
&\cdot\left(2+  \frac{2C_{HP}}{\min_{[0,1]}^3\eta}+\frac{1}{\delta}+\frac{1}{\delta}\frac{\max_{[0,1]}\eta+ C_{HP}}{\min_{[0,1]}^2 \eta} \right),
\end{aligned}
\end{equation}
being
$
C_{\ve_0, \beta, K, \eta}:=\ve_0 -  2\delta C_1
$
and $\Theta$ as in \eqref{Theta}.
\end{Theorem}
\begin{proof} 
As usual, let us start assuming that $y$ is a classical solution. 
If $y$ is the mild solution associated to the initial data $(y_0, y_1) \in \mathcal H_0$, consider a sequence $\{(y_0^n, y_1^n)\}_{n \in  \N} \in D(\mathcal A)$ that approximate $(y_0, y_1)$ and let $y^n$ be the classical solution of \eqref{wave} associated to $(y_0^n, y_1^n)$. With standard estimates coming from Theorem \ref{generator}, we can pass to the limit in \eqref{estiener} written for $y^n$ and obtain the desired result.

So, let $y$ be a classical solution. By \eqref{tondonew} and \eqref{Stimabo} we have
\[
\begin{aligned}
\ve_0&\int_s^T E_y(t)dt 
\le  4 \Theta E_y(s)  + \left( \frac{1}{\sigma(1)} + \frac{1}{\eta(1)}+\frac{\beta}{\eta(1)}+ \frac{K}{4} \right)E_y(s) \\
&+  \left(\frac{\beta^2}{\eta(1)}+ \frac{K\beta}{2}+\frac{\beta}{\eta(1)}+ \frac{K}{4}+ \frac{\beta\ve_0}{2}\right) \int_s^Ty^2(t,1)dt\\
&\le \left( 4\Theta + \frac{1}{\sigma(1)} + \frac{1}{\eta(1)}+\frac{\beta}{\eta(1)}+ \frac{K}{4} \right)E_y(s) \\
&+  \left(\frac{\beta^2}{\eta(1)}+ \frac{K\beta}{2}+\frac{\beta}{\eta(1)}+ \frac{K}{4} + \frac{\beta\ve_0}{2}\right) \left(2+  \frac{2C_{HP}}{\min_{[0,1]}^3\eta}+\frac{1}{\delta}+\frac{1}{\delta}\frac{\max_{[0,1]}\eta+ C_{HP}}{\min_{[0,1]}^2 \eta} \right)E_y(s) \\
&+  2\delta  C_1\int_s^T E_y(t) dt.
\end{aligned}
\]
Hence
\[
\begin{aligned}
&\left[ \ve_0 -  2\delta C_1\right]\int_s^T E_y(t)dt 
\le \left( 4\Theta + \frac{1}{\sigma(1)} + \frac{1}{\eta(1)}+\frac{\beta}{\eta(1)}+ \frac{K}{4} \right)E_y(s) \\
&+  \left(\frac{\beta^2}{\eta(1)}+ \frac{K\beta}{2}+\frac{\beta}{\eta(1)}+ \frac{K}{4}+ \frac{\beta\ve_0}{2}\right) \left(2+  \frac{2C_{HP}}{\min_{[0,1]}^3\eta}+\frac{1}{\delta}+\frac{1}{\delta}\frac{\max_{[0,1]}\eta+ C_{HP}}{\min_{[0,1]}^2 \eta} \right)E_y(s).
\end{aligned}
\]
Choosing $\ds
\delta\in \left( 0, \frac{\ve_0}{2C_1}\right),
$
we have
\[\ve_0 -  2\delta C_1>0.\]
Hence, we can apply Lemma \ref{Lemma4.6} below, obtaining
\[
E_y(t) \le E_y(0) e^{1- \frac{t}{M}},
\]
where  $M$ is as in \eqref{M}.
\end{proof}
\begin{Lemma}\label{Lemma4.6}
Assume that $E: [0, +\infty) \rightarrow [0, +\infty)$ is a nonincreasing function and that there is a constant $M>0$ such that
\[
\int_t^\infty E(s)ds \le ME(t), \quad \forall \; t \; [0, +\infty).
\]
Then
\[
E(t) \le E(0) e^{1- \frac{t}{M}},\quad \forall\; t \; [0, +\infty).
\]
\end{Lemma}
\section{Appendix}

For the readers' convenience, in this section we prove a technical lemma, used throughout the previous sections.
\begin{proof}[Proof of Lemma $\ref{lemmalimite}$]
1. We will actually prove an equivalent fact, usually appearing in integrations by parts, namely that
\[
\lim_{x \rightarrow 0} \eta (x)u(x) y'(x)=0.
\]
For this, consider the function
\[
z(x):=
\eta(x) u(x)  y'(x), \quad x \in (0,1].
\]
Observe that
\[
\int_0^1|z|dx \le C \|u\|_{L^2(0,1)}\|y'\|_{L^2(0,1)}.
\]
Moreover
\[
z'(x)= (\eta y')'u + \eta y'u'
\]
and, by H\"older's inequality, 
\[
\int_0^1|(\eta y')'u| dx= \int_0^1 \sqrt{\sigma}(\eta y')' \frac{u}{\sqrt{\sigma}} dx\le  \|\sigma(\eta y')'\|_{{\frac{1}{\sigma}}}\|u\|_{{\frac{1}{\sigma}}}
\]
and
\[
\int_0^1\eta|y'u'|dx\le C\|y'\|_{L^2(0,1)}\|u'\|_{L^2(0,1)}.
\]
Thus $z'$ is summable on $[0,1]$.  Hence $z \in W^{1,1}(0,1) \hookrightarrow C[0,1]$ and there exists
\[
\lim_{x \rightarrow 0}z(x)=\lim_{x \rightarrow 0} \eta(x) u(x)  y'(x)=L \in \R.
\]
We will prove that $L=0$.  If $L \neq 0$ there would exist a  neighborhood $\cal I$ of $0$ such that
\[
\frac{|L|}{2} \le |\eta y'u|,
\]
for all $x \in \cal I$;
but, by  H\"older's inequality,
\[
|u(x)|\le \int_0^x |u'(t) |dt \le  \sqrt{x} \|u'\|_{L^2(0,1)}.
\]
Hence
\[
\frac{|L|}{2} \le |\eta y'u| \le \|\eta\|_{\infty} |y'| \sqrt{x} \|u'\|_{L^2(0,1)}
\]
for all $x \in \cal I$. This would imply that
\[
|y'| \ge \frac{|L|}{2 \|u'\|_{L^2(0,1)} \|\eta\|_{\infty} \sqrt{x} }
\]
in contrast to the fact that $ y' \in L^2(0,1)$.

Hence $L=0$ and the conclusion follows.

2. It is enough to write
\[
xu'(\eta u')'=\frac{xu'}{\sqrt{\sigma}}\sqrt{\sigma}(\eta u')'
\]
and notice that, by Hypothesis \ref{Ass0},
\[
\left|\frac{xu'}{\sqrt{\sigma}}\right|^2\leq \|\eta\|_\infty \frac{x^2}{a(x)}(u')^2\leq C\|\eta\|_\infty(u')^2\in L^1(0,1),
\]
for a positive constant $C$,
and
\[
(\sqrt{\sigma}(\eta u')')^2=\frac{(Au)^2}{\sigma}\in L^1(0,1).
\]

3. 
Set $z(x)=x\eta(x)(u'(x))^2$. Of course, $z\in L^1(0,1)$. Moreover,
\[
z'=\eta(u')^2+x\eta'(u')^2+2x\eta u'u''=2xu'(\eta u')'-x\eta'(u')^2+ \eta (u')^2.
\]
By the previous point, $xu'(\eta u')'\in L^1(0,1)$, while
\[
\left|x\eta'(u')^2\right|=\left|x\eta \frac{b}{a}(u')^2\right|\leq C(u')^2
\]
by \eqref{defM}. Clearly, $\eta (u')^2 \in L^1(0,1)$, thus $z'\in L^1(0,1)$, so that $z\in W^{1,1}(0,1)$ and there exists $\lim_{x\to 0}z(x)=L\in\R$. If $L\neq0$, sufficiently close to $x=0$ we would have that
\[
(u'(x))^2\geq \frac{|L|}{2\eta x}\not\in L^1(0,1),
\]
while $u\in H^1(0,1)$. The conclusion follows since $\eta$ is bounded away from 0.

4. Proceed as above, observing that
\[
\left|x\eta'(u')^2\right|=\left|x\eta \frac{b}{a}(u')^2\right|\leq \left\| \frac{xb}{a}\right\|_{L^\infty(0,1)}\left\| \eta\right\|_{L^\infty(0,1)}(u')^2
\]
and that
\[
|xu'(\eta u')'|= \left| \frac{x}{\sqrt{a}}\sqrt{\eta}u' \right|\cdot\left| \sqrt{\sigma}(\eta u')'\right|.
\]

5. Set $\ds z:= \frac{x}{a}u^2(x)$. Then $z \in L^1(0,1)$. Indeed
\[
\int_0^1\frac{x}{a} u^2(x)dx \le \frac{1}{\min_{[0,1]}\eta} \int_0^1 \frac{u^2}{\sigma}dx.
\]
Moreover
\[
z'= \frac{u^2}{a}+ 2\frac{xuu'}{a} - \frac{a'x}{a^2}u^2;
\]
thus, for a suitable $\ve>0$ given by Hypothesis \ref{Ass0},
\[
\begin{aligned}
\int_0^\ve|z'|dx &\le \frac{1}{\min_{[0,1]}\eta}\int_0^1 \frac{u^2}{\sigma}dx + 2 \left(\int_0^\ve \frac{x^2(u')^2}{a}dx\right)^{\frac{1}{2}}\left( \int_0^1 \frac{u^2}{a}dx\right)^{\frac{1}{2}} + K\int_0^1 \frac{u^2}{a}dx\\
&\le \frac{1+K}{\min_{[0,1]}\eta}\int_0^1 \frac{u^2}{\sigma}dx + \frac{2\ve^2}{a(\ve)\min_{[0,1]}\eta}\left(\int_0^1 (u')^2dx\right)^{\frac{1}{2}}\left( \int_0^1 \frac{u^2}{\sigma}dx\right)^{\frac{1}{2}}.
\end{aligned}
\]
This is enough to conclude that  $z \in W^{1,1}(0,1)$ and thus there exists  $\lim_{x\to 0}z(x)=L\in\R$. If $L\neq0$, sufficiently close to $x=0$ we would have that
\[
\frac{u^2(x)}{a}\geq \frac{|L|}{2x}\not\in L^1(0,1),
\]
while $\ds\frac{u^2}{a} \in L^1(0,1)$ (since $\ds\frac{u^2}{\sigma} \in L^1(0,1)$). 
\end{proof}

\section*{Acknowledgments}
The authors wish to thank the anonymous referee for her/his detailed report, which has improved the presentation of the paper.

\end{document}